\documentclass[11pt]{amsart}

\usepackage{amsmath, amsthm, amssymb, bbm, mathrsfs}

\calclayout

\usepackage[style=numeric, maxbibnames=99]{biblatex}
\usepackage[colorlinks]{hyperref}
\hypersetup{
  linkcolor=[rgb]{0.3,0.3,0.6},
  citecolor=[rgb]{0.2, 0.6, 0.2},
  urlcolor=[rgb]{0.6, 0.2, 0.2}
}

\newtheorem{lem}{Lemma}

\newcommand{\parset}{
	\setlength{\parskip}{3mm}
  	\setlength{\parindent}{0mm}}
	
\parset
	
  \renewcommand{\Pr}{\mbox{\rm Pr}}
  \newcommand{\Exp}{{\mathbb{E}}}
  \DeclareMathOperator{\Bern}{Bern}

  \newcommand{\R}{\mathbb{R}} 
  \newcommand{\N}{\mathbb{N}} 
  \newcommand{\Z}{\mathbb{Z}} 
  \newcommand{\pmset}[1]{\{-1,1\}^{#1}} 
  \newcommand{\bset}[1]{\{0,1\}^{#1}} 

  \newcommand{\one}{\mathbf{1}}

  \newcommand{\st}{:\,} 

  \newcommand{\eps}{\varepsilon}

  \newcommand{\poly}{\mbox{\rm poly}}
  \newcommand{\ceil}[1]{\lceil{#1}\rceil}
  \newcommand{\floor}[1]{\lfloor{#1}\rfloor}

  \newcommand{\beq}{\begin{equation}}
  \newcommand{\eeq}{\end{equation}}
  \newcommand{\beqn}{\begin{equation*}}
  \newcommand{\eeqn}{\end{equation*}}
  \newcommand{\beqr}{\begin{eqnarray}}
  \newcommand{\eeqr}{\end{eqnarray}}
  \newcommand{\beqrn}{\begin{eqnarray*}}
  \newcommand{\eeqrn}{\end{eqnarray*}}
  \newcommand{\bmline}{\begin{multline}}
  \newcommand{\emline}{\end{multline}}
  \newcommand{\bmlinen}{\begin{multline*}}
  \newcommand{\emlinen}{\end{multline*}}

  \theoremstyle{plain}
  \newtheorem{theorem}{Theorem}[section]
  
  \newtheorem{proposition}[theorem]{Proposition}
  
  \newtheorem{corollary}[theorem]{Corollary}
  
  \theoremstyle{definition}

  \theoremstyle{remark}

  \renewenvironment{proof}[1][]{
    	\begin{trivlist}
     	\item[\hspace{\labelsep}{\em\noindent Proof#1:\/}]}
     	{{\hfill$\Box$}
    	\end{trivlist}
  }
  
  \makeatletter
  \newcommand{\newreptheorem}[2]{\newtheorem*{rep@#1}{\rep@title}\newenvironment{rep#1}[1]{\def\rep@title{#2 \ref*{##1}}\begin{rep@#1}}{\end{rep@#1}}}
  \makeatother

  \newreptheorem{lemma}{Lemma}
  \newreptheorem{theorem}{Theorem}
  \newreptheorem{corollary}{Corollary}
  \newreptheorem{proposition}{Proposition}
  \newreptheorem{conjecture}{Conjecture}

\begin{document}
\title[]{On the threshold for Szemer\'edi's theorem with random differences}

\author{Jop Bri\"{e}t}
\address{CWI \& QuSoft, Science Park 123, 1098 XG Amsterdam, The Netherlands}
\email{j.briet@cwi.nl}

\author{Davi Castro-Silva}
\address{CWI \& QuSoft, Science Park 123, 1098 XG Amsterdam, The Netherlands}
\email{davisilva15@gmail.com}

\thanks{This work was supported by the Dutch Research Council (NWO) as part of the NETWORKS programme (grant no. 024.002.003). An extended abstract of this work appeared in the proceedings of EUROCOMB'23~\cite{BrietCastroSilva:2023}.}

\begin{abstract}
Using recent developments on the theory of locally decodable codes, we prove that the critical size for Szemer\'edi's theorem with random differences is bounded from above by $N^{1-\frac{2}{k} + o(1)}$ for length-$k$ progressions.
This improves the previous best bounds of $N^{1-\frac{1}{\ceil{k/2}} + o(1)}$ for all odd~$k$.
\end{abstract}

\maketitle

\section{Introduction}

Szemer\'edi~\cite{Szemeredi:1975} proved that dense sets of integers contain arbitrarily long arithmetic progressions, a result which has become a hallmark of additive combinatorics.
Multiple proofs of this result were found over the years, using ideas from combinatorics, ergodic theory and Fourier analysis over finite abelian groups.

Furstenberg's ergodic theoretic proof~\cite{Furstenberg:1977} opened the floodgates to a series of powerful generalizations.
In particular, it led to versions of Szemer\'edi's theorem where the common differences for the arithmetic progressions are restricted to very sparse sets.
We say that a set $D\subseteq [N]$ is \emph{$t$-intersective} if any positive-density set $A\subseteq [N]$ contains a $(t+1)$-term arithmetic progression with common difference in~$D$.
Szemer\'edi's theorem implies that, for large enough~$N_0$, the set $\{0,1,\dots,N_0\}$ is $t$-intersective for all $N \geq N_0$.
Non-trivial examples include a result of Bergelson and Leibman~\cite{BergelsonL:1996} showing that the perfect squares (and more generally, images of integer polynomials with zero constant term) are $t$-intersective for every $t$, and
a special case of a result of Wooley and Ziegler~\cite{WooleyZ:2012} showing the same for the prime numbers minus one.

The existence of such sparse intersective sets motivated the problem of showing whether, in fact, random sparse sets are typically intersective.
The task of making this quantitative falls within the scope of research on threshold phenomena.
We say that a property of subsets of~$[N]$, given by a family $\mathcal F \subseteq 2^{[N]}$, is \emph{monotone} if $A\in \mathcal F$ and $A\subseteq B\subseteq [N]$ imply $B\in \mathcal F$.
The \emph{critical size} $m^* = m^*(N)$ of a property is the least~$m$ such that a uniformly random $m$-element subset of~$[N]$ has the property with probability at least~$1/2$.
(This value exists if~$\mathcal F$ is non-empty and monotone, as this probability then increases monotonically with~$m$.)
A famous result of Bollob\'as and Thomason~\cite{BollobasT:1987} asserts that every monotone property has a threshold function;
this is to say that the probability 
\beqn
p(m) = \Pr_{A\in \binom{[N]}{m}}[A\in \mathcal F]
\eeqn 
spikes from $o(1)$ to $1 - o(1)$ when~$m$ increases from $o(m^*)$ to~$\omega(m^*)$.
In general, it is notoriously hard to determine the critical size of a monotone property.

This problem is also wide open for the property of being $t$-intersective, which is clearly monotone, and for which we denote the critical size by~$m^*_t(N)$.
Bourgain~\cite{Bourgain:1987} showed that the critical size for 1-intersective sets is given by $m_1^*(N) \asymp \log N$;
at present, this is the only case where precise bounds are known.
It has been conjectured~\cite{FrantzikinakisLW:2016} that $\log N$ is the correct bound for all fixed~$t$, and indeed no better lower bounds are known for $t \geq 2$.
It was shown by Frantzikinakis, Lesigne and Wierdl~\cite{FrantzikinakisLW:2012} and independently by Christ~\cite{Christ:2011} that 
\beq\label{eq:rootbound}
m_2^*(N) \ll N^{\frac{1}{2} + o(1)}.
\eeq
The same upper bound was later shown to hold for $m_3^*(N)$ by the first author, Dvir and Gopi~\cite{BrietDG:2019}.
More generally, they showed that 
\beq\label{eq:ceilbounds}
m_t^*(N) \ll N^{1 - \frac{1}{\ceil{(t+1)/2}} + o(1)},
\eeq
which improved on prior known bounds for all $t \geq 3$.
The appearance of the ceiling function in these bounds is due to a reduction for even $t$ to the case $t+1$.
The reason for this reduction originates from work on locally decodable error correcting codes~\cite{KerenidisW:2004}.
It was shown in~\cite{BrietDG:2019} that lower bounds on the block length of $(t+1)$-query locally decodable codes (LDCs) imply upper bounds on~$m_t^*$.
The bounds~\eqref{eq:ceilbounds} then followed directly from the best known LDC bounds;
see~\cite{BrietG:2020} for a direct proof of~\eqref{eq:ceilbounds}, however.

For the same reason, a recent breakthrough of Alrabiah et al.~\cite{AlrabiahGKM:2022} on 3-query LDCs immediately implies an improvement of~\eqref{eq:rootbound} to
\beqn
m_2^*(N) \ll N^{\frac{1}{3} + o(1)}.
\eeqn
For technical reasons, their techniques do not directly generalize to improve the bounds for $q$-query LDCs with $q \geq 4$.
Here, we use the ideas of~\cite{AlrabiahGKM:2022} to directly prove upper bounds on~$m_t^*$.
Due to the additional arithmetic structure in our problem, it is possible to simplify the exposition and, more importantly, apply the techniques to improve the previous best known bounds for all even~$t\geq 2$.

\begin{theorem}\label{thm:integers}
For every integer $t \geq 2$, we have that 
\beqn
m_t^*(N) \ll N^{1 - \frac{2}{t+1} + o(1)}.
\eeqn
\end{theorem}

The arguments presented here in fact work in greater generality, and hold for any finite additive group~$G$ whose size is coprime to~$t!$
(so as not to incur in divisibility issues when considering $(t+1)$-term arithmetic progressions).

Let~$G$ be a finite additive group, $t\geq 1$ be an integer and $\eps\in (0, 1)$.
We say that a set $S\subseteq G$ is \emph{$(t,\eps)$-intersective} if every subset $A\subseteq G$ of size $|A|\geq \eps |G|$ contains a $(t+1)$-term arithmetic progression with common difference in~$D$.
We denote the critical size for the property of being $(t,\eps)$-intersective in~$G$ by $m_{t,\eps}^*(G)$.
Our main result is the following:

\begin{theorem}\label{thm:main}
For every $t \geq 2$ and $\eps\in (0,1)$, there exists $C(t,\eps)>0$ such that
\beqn
m_{t,\eps}^*(G) \leq C(t,\eps) (\log |G|)^{2t+3} |G|^{1 - \frac{2}{t+1}}
\eeqn
for every additive group~$G$ whose size is coprime to~$t!$.
\end{theorem}

Note that Theorem~\ref{thm:integers} follows easily from this last result by embedding~$[N]$ into a group of the form~$\Z/p\Z$, where~$p$ is a prime between~$(t+1)N$ and~$2(t+1)N$.
We omit the standard details.

\section{Preliminaries}
\label{sec:prelim}

\subsection{Notation}

We write $\biguplus$ for a disjoint union.
Our (standard) asymptotic notation is defined as follows.
Given a parameter~$n$ which grows without bounds and a function $f:\R_+ \to \R_+$, we write:
$g(n) = o(f(n))$ to mean $g(n)/f(n)\to 0$; $g(n)=\omega(f(n))$ to mean $g(n)/f(n)\to \infty$; $g(n) \ll f(n)$ to mean that $g(n) \leq C f(n)$ holds for some constant~$C>0$ and all~$n$;
and $g(n) \asymp f(n)$ to mean both $g(n) \ll f(n)$ and $f(n) \ll g(n)$.
When the implied constant in the asymptotics depends on some parameter (say $\eps$), we indicate this by adding said parameter as a subscript in the asymptotic notation (replacing $\ll$ by $\ll_{\eps}$, say).

\subsection{Matrix norms and inequalities}

Our arguments will rely heavily on the analysis of high-dimensional matrices.
Here we recall the matrix inequalities which will be needed.

If $M\in \R^{d\times d}$ is a matrix, we define its operator norms
\begin{align*}
    \|M\|_2 &= \max \big\{ u^T M v:\: \|u\|_2 = \|v\|_2 = 1\big\} \\
    \|M\|_{\infty\to 1} &= \max \big\{ u^T M v:\: \|u\|_\infty = \|v\|_\infty = 1\big\}\\
    \|M\|_{1\to 1} &= \max \big\{u^T Mv:\: \|u\|_\infty = \|v\|_1 = 1\big\}.
\end{align*}
We will make use of the following simple inequalities:
\beqn
\|M\|_{\infty\to 1} \leq d \|M\|_{2}, \quad \|M\|_{\infty\to 1} \leq \sum_{i=1}^d \|M(i, \cdot)\|_1
\eeqn
and, when $M$ is symmetric,
\beqn
\quad \|M\|_2 \leq  \|M\|_{1\to 1}.
\eeqn
We will also use the following noncommutative version of Khintchine's inequality, which can be extracted from a result of Tomczak-Jaegermann~\cite{TJ:1974}:

\begin{theorem}\label{thm:Khintchine}
Let $n, d \geq 1$ be integers, and let $A_1, \dots, A_n$ be any sequence of $d\times d$ real matrices.
Then
\beqn
\Exp_{\sigma \in \pmset{n}} \bigg\| \sum_{i=1}^n \sigma_i A_i \bigg\|_2 \leq 10 \sqrt{\log d} \bigg(\sum_{i=1}^n \|A_i\|_2^2 \bigg)^{1/2}.
\eeqn
\end{theorem}

\subsection{Polynomial concentration bounds}

We will need a well-known concentration inequality for polynomials due to Kim and Vu~\cite{KimV:2000}, which requires the introduction of some extra notation.
Let $H = (V,E)$ be a hypergraph, where we allow for repeated edges (so~$E$ may be a multiset), and let $f:\bset{V}\to \R$ be the polynomial given by
\beq\label{eq:fkimvu}
f(x) = \sum_{e\in E}\prod_{v\in e}x_v.
\eeq
For a set $A\subseteq V$, define
\beqn
f_A(x) = \sum_{e\in E:\, A\subseteq e}\prod_{v\in e\setminus A}x_v,
\eeqn
where the monomial corresponding to the empty set is defined to be~1.
For $p\in (0,1)$, we say that $X$ is a \emph{$p$-Bernoulli random variable on $\bset{V}$},
denoted $X\sim\Bern(p)^V$, if its coordinates are all independent and each equals~1 with probability~$p$ (and equals~0 with probability~$1-p$).
For each $i\in \{0,1,\dots, |V|\}$, define
\beqn
\mu_i = \max_{A\in {V\choose i}}\Exp_{X\sim\Bern(p)^V} f_A(X).
\eeqn
Note that $\mu_0$ is just the expectation of~$f(X)$. Define also the quantities
\beqn
\mu = \max_{i\in \{0,1,\dots,|V|\}} \mu_i
\quad\quad\text{and}\quad\quad
\mu' =\max_{i\in \{1,2,\dots,|V|\}} \mu_i.
\eeqn


The polynomial concentration inequality of Kim and Vu is given as follows:

\begin{theorem} \label{thm:kimvu}
    For every $k\in \N$, there exist constants $C,C'>0$ such that the following holds.
    Let $H = (V,E)$ be an $n$-vertex hypergraph whose edges have size at most~$k$, and let~$f$ be given by~\eqref{eq:fkimvu}.
    Then, for any $\lambda >1$, we have
    \beqn
    \Pr\big[ |f(X) - \mu_0| > C\lambda^{k - \frac{1}{2}}\sqrt{\mu\mu'}] \leq C'\exp\big(-\lambda + (k-1)\log n\big).
    \eeqn
\end{theorem}

To suit our needs, we will use a slight variant of this result, which follows easily from it and the following basic proposition.

\begin{proposition}\label{prop:set_bern}
Let $f:\bset{n}\to \R_+$ be a monotone increasing function and $p\in (\frac{16}{n},1)$.
Then, for any integer $0\leq t\leq pn/2$, 
\beqn
\Exp_{S\in {[n]\choose t}}f(1_S) \leq \frac{1}{2}\, \Exp_{X\sim\Bern(p)^n}f(X).
\eeqn
\end{proposition}

\begin{proof}
By direct calculation,
\begin{align*}
    \Exp_{X\sim\Bern(p)^n}f(X) &= \sum_{i=0}^n p^i(1-p)^{n-i}\sum_{S\in {[n]\choose i}}f(1_S)\\
    &= \sum_{i=0}^n{n\choose i} p^i(1-p)^{n-i}\, \Exp_{S\in {[n]\choose i}} f(1_S)\\
    &\geq \sum_{i\geq t}{n\choose i} p^i(1-p)^{n-i}\, \Exp_{S\in {[n]\choose t}}f(1_S)\\
    &\geq \frac{1}{2}\,\Exp_{S\in {[n]\choose t}}f(1_S),
\end{align*}
where in the third  line we used monotonicity of~$f$ and the fourth line follows from the Chernoff bound.
\end{proof}

\begin{corollary}\label{cor:kimvu}
For every $k\in \N$, there exist constants $C, C'>0$ such that the following holds.
Let $H = (V,E)$ be an $n$-vertex hypergraph whose edges have size at most~$k$, let~$f$ be given as in~\eqref{eq:fkimvu} and let $p\in (\frac{16}{n},1)$.
Then, for any integer $0\leq t\leq pn/2$, we have
\beqn
\Pr_{S\in {V\choose t}}\big[ f(1_S) \geq C(\log n)^{k - \frac{1}{2}}\mu] \leq \frac{C'}{n^4}.
\eeqn
\end{corollary}

\begin{proof}
For a sufficiently large constant~$C =C(k)>0$, let $g:\bset{n}\to \bset{}$ be the indicator function
\beqn
g(1_S) = \one\big[ f(1_S) \geq C(\log n)^{k - \frac{1}{2}}\mu].
\eeqn
Since~$f$ is monotone, so is~$g$.
Setting~$\lambda = (3+k)\log n$, it follows from Theorem~\ref{thm:kimvu} that
\beqn
\Exp_{X\sim\Bern(p)^n}\,g(X) \leq \frac{C'}{n^4}.
\eeqn
The result now follows from Proposition~\ref{prop:set_bern}.
\end{proof}

\section{The main argument}

In this section we will prove Theorem~\ref{thm:main}, our main result.
It will be more convenient to shift attention from the degree $t$ of intersectivity to the length $k:=t+1$ of the associated arithmetic progressions.

Fix then an integer $k\geq 3$ for the length of the progressions and a positive parameter $\eps>0$.
Let~$G$ be an additive group with $N$ elements, where~$N$ is coprime to~$(k-1)!$ and is assumed to be sufficiently large relative to~$k$ and~$\eps$ for our arguments to hold.
Recall that we wish to show that
$m_{k-1, \eps}^*(G) \ll_{k,\eps} (\log N)^{2k+1} N^{1-\frac{2}{k}}$.

Instead of considering random intersective sets, it will be simpler to consider random \emph{intersective sequences}, where a sequence in~$G^m$ is $(k-1, \eps)$-intersective if the set of its distinct elements is.
Clearly, the probability that a uniformly random $m$-element sequence is $(k-1, \eps)$-intersective is at most the probability that a uniform $m$-element set is.
Since we are interested in proving upper bounds on the critical size, it suffices to bound the minimal~$m$ such that a random sequence in~$G^m$ is $(k-1, \eps)$-intersective with probability at least~$1/2$.

\subsection{Reducing to an inequality about sign averages}

Given a sequence of differences $D = (d_1, \dots, d_m) \in G^m$ and some set~$A\subseteq G$, let~$\Lambda_D(A)$ be the normalized count of $k$-APs with common difference in~$D$ which are contained in~$A$:
\beqn
\Lambda_D(A) = \Exp_{i\in [m]} \Exp_{x\in G} \prod_{\ell=0}^{k-1} A(x + \ell d_i).
\eeqn
Similarly, we denote by~$\Lambda_G(A)$ the proportion of all $k$-APs which are contained in~$A$:
\beqn
\Lambda_G(A) = \Exp_{d\in G} \Exp_{x\in G} \prod_{\ell=0}^{k-1} A(x + \ell d).
\eeqn

For the rest of the proof we will assume that $m\in[N]$ is an integer for which
\beq\label{eq:Sz_failure}
\Pr_{D\in G^m}\big(\exists A\subseteq G:\: |A|\geq \eps N,\, \Lambda_D(A) = 0\big) > 1/2,
\eeq
meaning that a random $m$-element sequence is unlikely to be $(k-1, \eps)$-intersective.
Note that this inequality holds whenever $m < m_{k-1,\eps}^*(G)$.
Our first task will be to reduce equation~\eqref{eq:Sz_failure} to an inequality concerning sign averages (equation~\eqref{eq:even_XOR} below), for which we have more tools at our disposal.

By a suitable generalization of Szemer\'edi's theorem, we know that
\beq\label{eq:Szemeredi}
\Lambda_G(A) \gg_{k,\eps} 1 \quad \text{for all $A\subseteq G$ with $|A|\geq \eps N$.}
\eeq
This can be proven, for instance, by using the \emph{hypergraph removal lemma} of Gowers~\cite{Gowers:2007} and Nagle, R\"odl, Schacht and Skokan~\cite{RodlS:2004, NagleRS:2006}.
Noting that $\Exp_{D'\in G^m} \Lambda_{D'}(A) = \Lambda_G(A)$, by combining inequalities~\eqref{eq:Sz_failure} and~\eqref{eq:Szemeredi} we conclude that
\beq\label{eq:outlaw}
\Exp_{D\in G^m} \max_{A\subseteq G:\: |A|\geq \eps N} \big|\Lambda_D(A) - \Exp_{D'\in G^m} \Lambda_{D'}(A)\big| \gg_{k,\eps} 1.
\eeq
Below, we will no longer need the condition that $|A|\geq \eps N$ in maxima over $A\subseteq G$.
This positive density assumption is only used through~\eqref{eq:Szemeredi}.

We next apply a simple symmetrization argument given in~\cite[page~8690]{BrietG:2020} to write this in a more convenient form:

\begin{lem}[Symmetrization]\label{lem:symmetrization}
Let~$c>0$, and suppose that
\beqn
\Exp_{D\in G^m} \max_{A\subseteq G} \big|\Lambda_D(A) - \Exp_{D'\in G^m} \Lambda_{D'}(A)\big| \geq c.
\eeqn
Then
\beqn
\Exp_{D\in G^m} \Exp_{\sigma\in \pmset{m}} \max_{A\subseteq G} \bigg|\Exp_{i\in [m]} \Exp_{x\in G} \,\sigma_i \prod_{\ell=0}^{k-1} A(x + \ell d_i)\bigg| \geq \frac{c}{2}.
\eeqn
\end{lem}

\begin{proof}
With slight abuse of notation, for each $d\in G$ and $A\subseteq G$, denote
\beqn
\Lambda_d(A) = \Exp_{x\in G}\prod_{\ell=0}^{k-1}A(x + \ell d).
\eeqn
Then, by the triangle inequality, the first expression in Lemma~\ref{lem:symmetrization} is bounded from above by
\beqn
\Exp_{D,D'\in G^m} \max_{A\subseteq G} \Big|\frac{1}{m}\sum_{i=1}^m \big(\Lambda_{d_i}(A) - \Lambda_{d_i'}(A)\big)\Big|,
\eeqn
where $D$ and $D'$ are independent and uniformly distributed.
Since the terms $\Lambda_{d_i}(A) - \Lambda_{d_i'}(A)$ are independent and symmetrically distributed, this expectation is unchanged if each of these terms is multiplied by an arbitrary sign.
In particular, this expectation equals
\beqn
\Exp_{D,D'\in G^m} \Exp_{\sigma\in \pmset{m}}\max_{A\subseteq G} \Big|\frac{1}{m}\sum_{i=1}^m \sigma_i\big(\Lambda_{d_i}(A) - \Lambda_{d_i'}(A)\big)\Big|.
\eeqn
Using the triangle inequality again, and the fact that $D$ and $D'$ have the same distribution, we get that this is bounded from above by
\beqn
2\Exp_{D\in G^m} \Exp_{\sigma\in \pmset{m}}\max_{A\subseteq G} \Big|\frac{1}{m}\sum_{i=1}^m \sigma_i\Lambda_{d_i}(A) \Big|.
\eeqn
This proves the result.
\end{proof}

The appearance of the expectation over signs $\sigma\in \pmset{m}$ is crucial to our arguments.
By an easy multilinearity argument, we can replace the set~$A\subseteq G$ (which can be seen as a vector in~$\{0, 1\}^G$) by a vector $Z\in \pmset{G}$.
In combination with~\eqref{eq:outlaw} and Lemma~\ref{lem:symmetrization}, this gives
\beq\label{eq:even_XOR}
\Exp_{D\in G^m} \Exp_{\sigma\in \pmset{m}} \max_{Z\in \pmset{G}} \bigg|\Exp_{i\in [m]} \Exp_{x\in G} \,\sigma_i \prod_{\ell=0}^{k-1} Z(x + \ell d_i)\bigg| \gg_{k,\eps} 1.
\eeq
The change from~$\{0, 1\}^G$ to~$\pmset{G}$ is a convenient technicality so that we can ignore terms which get squared in a product.

\subsection{Dealing with an odd number of terms}

The last inequality~\eqref{eq:even_XOR} is what we need to prove the result for even values of~$k$ using the arguments we will outline below.
For odd values of~$k$, however, this inequality is unsuited due to the odd number of factors inside the product.
The main idea from~\cite{AlrabiahGKM:2022} to deal with this case is to apply a ``Cauchy-Schwarz trick'' to
obtain a better suited inequality:

\begin{lem}[Cauchy-Schwarz trick]\label{lem:CS}
Let~$c>0$, and suppose~$m\geq 2/c^2$ is an integer for which
\beqn
\Exp_{D\in G^m} \Exp_{\sigma\in \pmset{m}} \max_{Z\in \pmset{G}} \bigg|\Exp_{i\in [m]} \Exp_{x\in G} \,\sigma_i \prod_{\ell=0}^{k-1} Z(x + \ell d_i)\bigg| \geq c.
\eeqn
Then there exists a partition~$[m] = L \,\uplus\, R$ such that
\beqn
\Exp_{D\in G^m} \Exp_{\substack{\sigma\in \pmset{L}\\ \tau\in\pmset{R}}} \max_{Z\in \pmset{G}} \sum_{\substack{i\in L\\ j\in R}} \sum_{x\in G} \sigma_i \tau_j \prod_{\ell=1}^{k-1} Z(x + \ell d_i) Z(x + \ell d_j) \geq \frac{c^2 m^2 N}{8}.
\eeqn
\end{lem}

\begin{proof}
By Cauchy-Schwarz, for any $Z\in \pmset{G}$ we have
\begin{align*}
    \bigg|\Exp_{i\in [m]} \Exp_{x\in G} \,\sigma_i \prod_{\ell=0}^{k-1} Z(x + \ell d_i)\bigg|^2
    &= \bigg|\Exp_{x\in G} \,Z(x) \cdot \bigg(\Exp_{i\in [m]} \sigma_i \prod_{\ell=1}^{k-1} Z(x + \ell d_i)\bigg)\bigg|^2 \\
    &\leq \big(\Exp_{x\in G} \,Z(x)^2\big) \Exp_{x\in G} \bigg(\Exp_{i\in [m]} \sigma_i \prod_{\ell=1}^{k-1} Z(x + \ell d_i)\bigg)^2 \\
    &= \Exp_{x\in G} \Exp_{i, j\in [m]} \,\sigma_i  \sigma_j \prod_{\ell=1}^{k-1} Z(x + \ell d_i) Z(x + \ell d_j).
\end{align*}
Applying Cauchy-Schwarz again, we conclude from our assumption that
\begin{align*}
    c^2 &\leq \Exp_{D\in G^m} \Exp_{\sigma\in \pmset{m}} \max_{Z\in \pmset{G}} \bigg|\Exp_{i\in [m]} \Exp_{x\in G} \,\sigma_i \prod_{\ell=0}^{k-1} Z(x + \ell d_i)\bigg|^2 \\
    &\leq \Exp_{D\in G^m} \Exp_{\sigma\in \pmset{m}} \max_{Z\in \pmset{G}} \Exp_{x\in G} \Exp_{i, j\in [m]} \,\sigma_i  \sigma_j \prod_{\ell=1}^{k-1} Z(x + \ell d_i) Z(x + \ell d_j).
\end{align*}

Now consider a uniformly random partition~$[m] = L \,\uplus\, R$, so that for any $i, j\in [m]$ with $i\neq j$ we have $\Pr_{L, R}(i\in L,\, j\in R) = 1/4$;
then
\begin{align*}
    \Exp_{i, j\in [m]} &\,\sigma_i  \sigma_j \prod_{\ell=1}^{k-1} Z(x + \ell d_i) Z(x + \ell d_j) \\
    &= \frac{1}{m^2} \sum_{\substack{i, j = 1 \\ i\neq j}}^{m} \sigma_i  \sigma_j \prod_{\ell=1}^{k-1} Z(x + \ell d_i) Z(x + \ell d_j) + \frac{1}{m^2} \sum_{i=1}^{m} \sigma_i^2 \prod_{\ell=1}^{k-1} Z(x + \ell d_i)^2 \\
    &= \frac{4}{m^2} \Exp_{L, R} \sum_{i\in L, j\in R} \sigma_i  \sigma_j \prod_{\ell=1}^{k-1} Z(x + \ell d_i) Z(x + \ell d_j) + \frac{1}{m}.
\end{align*}
It follows that
\beqn
    c^2 \leq \frac{1}{m} + \frac{4}{m^2} \Exp_{L, R} \Exp_{D\in G^m} \Exp_{\sigma\in \pmset{m}} \max_{Z\in \pmset{G}} \Exp_{x\in G} \sum_{i\in L, j\in R} \sigma_i  \sigma_j \prod_{\ell=1}^{k-1} Z(x + \ell d_i) Z(x + \ell d_j).
\eeqn
Using that $m\geq 2/c^2$, we conclude there exists a choice of partition $[m] = L \,\uplus\, R$ satisfying the conclusion of the lemma.
\end{proof}

From now on we assume that~$k$ is odd, and write~$k=2r+1$.\footnote{The even case is similar but simpler. We focus on the odd case here because this is where we get new bounds.} 
For $i, j\in [m]$, denote $P_i(x) = \{x+d_i, x+2d_i, \dots, x+2rd_i\}$ and $P_{ij}(x) = P_i(x) \cup P_j(x)$, where we hide the dependence on the difference set~$D$ for ease of notation.
From inequality~\eqref{eq:even_XOR} and Lemma~\ref{lem:CS} we conclude that
\beq\label{eq:odd_XOR}
\Exp_{D\in G^m} \Exp_{\substack{\sigma\in \pmset{L}\\ \tau\in\pmset{R}}} \max_{Z\in \pmset{G}} \sum_{\substack{i\in L\\ j\in R}} \sum_{x\in G} \sigma_i \tau_j \prod_{y\in P_{ij}(x)} Z(y) \gg_{k,\eps} m^2 N,
\eeq
where~$(L, R)$ is a suitable partition of the index set~$[m]$ and we assume (without loss of generality) that~$m$ is sufficiently large depending on~$\eps$ and~$k$.

From inequality~\eqref{eq:odd_XOR} it follows that we can fix a ``good'' set~$D\in G^m$ satisfying
\beq\label{eq:large_exp}
\Exp_{\substack{\sigma\in \pmset{L}\\ \tau\in\pmset{R}}} \max_{Z\in \pmset{G}} \sum_{\substack{i\in L\\ j\in R}} \sigma_i \tau_j \sum_{x\in G} \prod_{y\in P_{ij}(x)} Z(y) \gg_{k,\eps} m^2 N
\eeq
and for which we have the technical conditions 
\begin{align}
    \big| \big\{i\in L, j\in R:\: |P_{ij}(0)| \neq 4r\big\} \big| &\ll_{k,\eps} m^2/N \quad \text{and} \label{eq:distinct} \\
    \max_{x\neq 0} \sum_{i=1}^m \sum_{\ell = -2r}^{2r} \one\{\ell d_i = x\} &\ll \log N, \label{eq:multiplicity}
\end{align}
which are needed to bound the probability of certain bad events later on.
Indeed:
\begin{itemize}
    \item Denote by $X(D)$ the expression on the left-hand side of~\eqref{eq:large_exp} for a given sequence $D\in G^m$.
    Then $X(D) \leq m^2N$ always holds, while by equation~\eqref{eq:odd_XOR} we have $\Exp_{D\in G^m} X(D) \geq c_{k,\eps} m^2N$ for some constant $c_{k,\eps} > 0$.
    It follows that
    \beqn
        \Pr\big[X(D) \geq c_{k,\eps} m^2N/2\big] > c_{k,\eps}/2.
    \eeqn
    
    \item For $\ell, \ell'\in [2r]$ and independent uniform $d_i, d_j\in G$, we have that
    \beqn
        \Pr[\ell d_i = \ell' d_j] = 1/N.
    \eeqn
    The expectation of the left-hand side of~\eqref{eq:distinct} (taken with respect to $D \in G^m$) is then
    \beqn
    \sum_{\substack{i\in L\\ j\in R}} \Pr\big[|P_{ij}(0)|\neq 4r\big] \leq \sum_{\substack{i\in L\\ j\in R}} \sum_{\ell, \ell' = 1}^{2r} \Pr[\ell d_i = \ell' d_j] \leq \frac{r^2 m^2}{N},
    \eeqn
    and thus by Markov's inequality
    \beqn
        \Pr\bigg[\big| \big\{i\in L, j\in R:\: |P_{ij}(0)| \neq 4r\big\} \big| \leq \frac{4}{c_{k,\eps}} \frac{r^2 m^2}{N}\bigg] \geq 1- \frac{c_{k,\eps}}{4}.
    \eeqn
    
    \item For a fixed~$x\neq 0$, the inner sum in~\eqref{eq:multiplicity} is an indicator random variable that equals~1 with probability $4r/N$.
    Since these random variables are independent for different~$i\in [m]$, the Chernoff bound implies that
    \beqn
        \Pr \bigg[\sum_{i=1}^m \sum_{\ell = -2r}^{2r} \one\{\ell d_i = x\} \geq (1+\delta) \frac{4rm}{N}\bigg] \leq \bigg(\frac{e^{\delta}}{(1+\delta)^{1+\delta}}\bigg)^{4rm/N}.
    \eeqn
    Setting $1+\delta = N\log N/(rm)$, the union bound over all $x\neq 0$ gives that
    \beqn
        \Pr \bigg[\max_{x\neq 0} \sum_{i=1}^m \sum_{\ell = -2r}^{2r} \one\{\ell d_i = x\} \leq 4\log N\bigg] \geq 1 - \frac{1}{N^3} > 1 - \frac{c_{k,\eps}}{4}.
    \eeqn
\end{itemize}
By the union bound, all three conditions above will hold simultaneously with positive probability, as wished.

\subsection{Reducing to a matrix inequality problem}

The next key idea is to construct matrices~$M_{ij}$ for which the quantity
\beq\label{eq:infty_one}
\Exp_{\substack{\sigma\in \pmset{L}\\ \tau\in\pmset{R}}} \bigg\| \sum_{i\in L, j\in R} \sigma_i \tau_j M_{ij} \bigg\|_{\infty\to 1}
\eeq
is related to the expression on the left-hand side of inequality~\eqref{eq:large_exp}.
The reason for doing so is that this allows us to use strong \emph{matrix concentration inequalities}, which can be used to obtain a good upper bound on the expectation~\eqref{eq:infty_one};
this in turn translates to an upper bound on~$m$ as a function of~$N$, which is our goal.
Such uses of matrix inequalities go back to work of Ben-Aroya, Regev and de~Wolf~\cite{BenAroyaRdW:2008}, in turn inspired by work of Kerenidis and de Wolf~\cite{KerenidisW:2004} (see also~\cite{Briet:2016}).

The matrices we will construct are indexed by sets of a given size~$s$, where (with hindsight) we choose $s = \floor{N^{1-2/k}}$.
Recall that $k=2r+1$.
For $i\in L$, $j\in R$, define the matrix $M_{ij} \in \R^{\binom{G}{s} \times \binom{G}{s}}$ by
$$M_{ij}(S, T) = \sum_{x\in G} \one\big\{ |S\cap P_i(x)| = |S\cap P_j(x)| = r,\, S\triangle T = P_{ij}(x) \big\}$$
if $|P_{ij}(0)| = 4r$, and $M_{ij}(S, T) = 0$ if $|P_{ij}(0)| \neq 4r$;
note that, despite the asymmetry in their definition, these matrices are in fact symmetric.
We will next deduce from inequality~\eqref{eq:large_exp} a lower bound on the expectation~\eqref{eq:infty_one}.

For a vector~$Z\in \pmset{G}$, denote by~$Z^{\odot s}\in \pmset{\binom{G}{s}}$ the ``lifted'' vector given by
\beqn
Z^{\odot s}(S) = \prod_{y\in S} Z(y) \quad \text{for all } S\in \binom{G}{s}.
\eeqn
If $|P_{ij}(0)| = 4r$, then for all $Z\in \pmset{G}$ we have
\begin{align}
    \sum_{S, T \in \binom{G}{s}} M_{ij}(S, T) Z^{\odot s}(S) Z^{\odot s}(T)
    &= \sum_{S, T \in \binom{G}{s}} M_{ij}(S, T) \prod_{y\in S\triangle T} Z(y) \nonumber\\
    &= \sum_{x\in G} \sum_{S\in \binom{G}{s}} \one\big\{ |S\cap P_i(x)| = |S\cap P_j(x)| = r\big\} \prod_{y\in P_{ij}(x)} Z(y) \nonumber\\
    &= \binom{2r}{r}^2 \binom{N-4r}{s-2r} \sum_{x\in G} \prod_{y\in P_{ij}(x)} Z(y),\label{eq:MSTZbound}
\end{align}
since there are $\binom{2r}{r}^2 \binom{N-4r}{s-2r}$ ways of choosing a set $S\in \binom{G}{s}$ satisfying $|S\cap P_i(x)| = |S\cap P_j(x)| = r$
and, once such a set~$S$ is chosen, there is only one set $T\in \binom{G}{s}$ for which $S\triangle T = P_{ij}(x)$.
It follows that
\begin{align*}
    \Exp_{\substack{\sigma\in \pmset{L}\\ \tau\in\pmset{R}}} &\bigg\| \sum_{i\in L, j\in R} \sigma_i \tau_j M_{ij} \bigg\|_{\infty\to 1} \\
    &\geq \Exp_{\substack{\sigma\in \pmset{L}\\ \tau\in\pmset{R}}} \max_{Z\in \pmset{G}} \sum_{S, T \in \binom{G}{s}} \sum_{i\in L, j\in R} \sigma_i \tau_j M_{ij}(S, T) Z^{\odot s}(S) Z^{\odot s}(T) \\
    &= \Exp_{\substack{\sigma\in \pmset{L}\\ \tau\in\pmset{R}}} \max_{Z\in \pmset{G}} \binom{2r}{r}^2 \binom{N-4r}{s-2r} \sum_{\substack{i\in L, j\in R \\ |P_{ij}(0)| = 4r}} \sigma_i \tau_j \sum_{x\in G} \prod_{y\in P_{ij}(x)} Z(y);
\end{align*}
combining this with inequalities~\eqref{eq:large_exp} and~\eqref{eq:distinct}, we conclude the lower bound
\beq\label{eq:lower_bound}
    \Exp_{\substack{\sigma\in \pmset{L}\\ \tau\in\pmset{R}}} \bigg\| \sum_{i\in L, j\in R} \sigma_i \tau_j M_{ij} \bigg\|_{\infty\to 1} \gg_{k,\eps} \binom{N-4r}{s-2r} m^2 N.
\eeq

\subsection{Applying a Khintchine-type inequality}

Now we need to compute an upper bound for the expectation above.
The main idea here is to use the non-commutative version of Khintchine's inequality given in Theorem~\ref{thm:Khintchine}.
Intuitively, this inequality shows that the sum in the last expression incurs many cancellations due to the presence of the random signs~$\sigma_i$, and thus the expectation on the left-hand side of~\eqref{eq:lower_bound} is much smaller than one might expect.

To apply Theorem~\ref{thm:Khintchine}, it is better to collect the matrices~$M_{ij}$ into groups and use only one half of the random signs~$\sigma_i$ (another idea from~\cite{AlrabiahGKM:2022}).
For $i\in L$, $\tau \in \{-1, 1\}^R$, we define the matrix
\beqn
M_i^{\tau} = \sum_{j\in R} \tau_j M_{ij}.
\eeqn
We will then provide an upper bound for the expression
\beqn 
\max_{\tau \in \pmset{R}} \Exp_{\sigma \in \pmset{L}} \bigg\| \sum_{i\in L} \sigma_i M_i^{\tau} \bigg\|_{\infty\to 1}
\eeqn
which is itself an upper bound for the expectation in~\eqref{eq:lower_bound}.

Towards this goal, we will prune the matrices~$M^\tau_i$ by removing all rows and columns whose $\ell_1$-weight significantly exceeds the average.
By symmetry and non-negativity of these matrices, the $\ell_1$-weight of a row or column indexed by a set~$S\in {G\choose s}$ is bounded by
\begin{align*}
 \sum_{T\in \binom{G}{s}} \bigg|\sum_{j\in R} \tau_j M_{ij}(S, T)\bigg|
 &\leq
 \sum_{T\in \binom{G}{s}} \sum_{j\in R} M_{ij}(S, T)\\&=
 \sum_{\substack{j\in R \\ |P_{ij}(0)| = 4r}} \sum_{x\in G} \one\big\{ |S\cap P_i(x)| = |S\cap P_j(x)| = r\big\}.
\end{align*}
To show that pruning makes little difference to the final bounds, we show that only a small proportion of the rows and columns have large $\ell_1$-weight.
To this end, let $U$ be a uniformly distributed $\binom{G}{s}$-valued random variable and, for each $i\in L$, define the random variable corresponding to the last expression above,
\beqn
X_i :=  \sum_{\substack{j\in R \\ |P_{ij}(0)| = 4r}} \sum_{x\in G} \one\big\{ |U\cap P_i(x)| = |U\cap P_j(x)| = r\big\}.
\eeqn
The calculation done in~\eqref{eq:MSTZbound}, with $Z$ the all-ones vector, shows that
\beqn
    \Exp[X_i] = \frac{1}{\binom{N}{s}} \sum_{\substack{j\in R \\ |P_{ij}(0)| = 4r}} \binom{2r}{r}^2 \binom{N-4r}{s-2r} N
    \ll_k \frac{1}{\binom{N}{s}} \binom{N-4r}{s-2r} mN.
\eeqn
Since $s=\lfloor N^{1-2/k}\rfloor$, we have that $\binom{N-4r}{s-2r}/\binom{N}{s} \ll_k (s/N)^{2r} \asymp N^{-(2-2/k)}$ and thus
\beq
    \Exp[X_i] \ll_k  \frac{m}{N^{1 - 2/k}}. \label{eq:mu0bound}
\eeq
The following lemma gives an upper-tail estimate on~$X_i$, provided~$m$ is sufficiently large.


\begin{lem}\label{lem:prune}
    Suppose that $m \geq N^{1 -2/k}$. Then, for every $i\in L$, we have that
    \beqn
        \Pr\Big[X_i \geq (\log N)^k \frac{m}{N^{1 - 2/k}}\Big] \leq \frac{1}{N^4}.
    \eeqn
\end{lem}

\begin{proof}
Fix an $i\in L$.
Consider the hypergraph~$H_i$ on vertex set~$G$ and with edge set
\beqn
E(H_i) = \biguplus_{\substack{j\in R \\ |P_{ij}(0)| = 4r}} \biguplus_{x\in G} \binom{P_i(x)}{r} \times \binom{P_j(x)}{r},
\eeqn
and let~$f:\R^G\to \R$ be the polynomial associated with~$H_i$ as in~\eqref{eq:fkimvu},
\beqn
f(t) = \sum_{e\in E(H_i)}\prod_{v\in e} t_v.
\eeqn
Note that $X_i = f(1_U)$, where~$U$ is uniformly distributed over $\binom{G}{s}$ and~$1_U \in \R^G$ denotes its (random) indicator vector.

For each $0\leq \ell \leq 2r$, we wish to bound the quantity
\beqn
\mu_\ell := \max_{A\in \binom{G}{\ell}} \Exp_{t\sim \Bern(s/N)^G} f_A(t).
\eeqn
(Recall the notation introduced in Section~\ref{sec:prelim}.)
By~\eqref{eq:mu0bound}, we have that $\mu_0 \ll_k mN^{-(1 - 2/k)}$.
For a set $A\in {G\choose \ell}$, define its degree in~$H_i$ by
\beqn
\deg(A) = |\{e\in E(H_i) \st e\supseteq A\}|,
\eeqn
where we count multiplicities of repeated edges.
Note that for any $B\subseteq A$, we have that $\deg(A) \leq \deg(B)$.
Then,
\beqn
\mu_\ell = \max_{A\in \binom{G}{\ell}} \Big(\frac{s}{N}\Big)^{2r-\ell}\deg(A).
\eeqn
For any $v\in G$, we have that $\deg(v) \ll_k m$, since~$v$ is contained in~$O_k(1)$ arithmetic progressions of length~$k$ with a fixed common difference.
It follows that for $\ell \in [r]$, we have that
\beqn
\mu_\ell \:\leq\: 
\Big(\frac{s}{N}\Big)^{2r-\ell} \max_{v\in G}\deg(v)
\:\ll_k\: 
m N^{-2r/(2r+1)} 
\:=\: \frac{m}{N^{1 - 1/k}}.
\eeqn

Let $A\subseteq G$ be a set of size $\ell\in \{r+1,\dots,2r\}$ and 
\beqn
e\in \binom{P_i(x)}{r} \times \binom{P_j(x)}{r}
\eeqn
 be an edge of~$E(H_i)$ that contains~$A$.
 By the Pigeonhole principle,~$A$ contains an element~$a\in P_i(x)$ and an element~$b\in P_j(x)$.
 Knowing~$a$ limits~$x$ to a set of size at most~$2r$.
 Moreover, it follows from~\eqref{eq:multiplicity} that for each~$x$, there are at most~$O_k(\log N)$ possible values of~$j\in R$ such that $b\in P_j(x)$.
 Therefore,
 \beqn
\mu_\ell \ll_k \Big(\frac{s}{N}\Big)^{2r-\ell} \log N \leq \log N.
 \eeqn

 Using our assumption on~$m$, it follows that for each $\ell\in \{0,\dots, 2r\}$, we have that $\mu_\ell \ll_k mN^{-(1 - 2/k)}\log N$.
 The result now follows directly from Corollary~\ref{cor:kimvu}.
 \end{proof}

Lemma~\ref{lem:prune} shows that for each matrix~$M_i^\tau$, at most an $N^{-4}$ fraction of all rows and columns have $\ell_1$-weight exceeding $(\log N)^k mN^{-(1 - 2/k)}$.
Now define~$\widetilde{M}_i^{\tau}$ as the `pruned' matrix obtained from~$M_i^{\tau}$ by zeroing out all such heavy rows and columns.
Note that~$\widetilde{M}_i^{\tau}$ is symmetric, and so
\beqn
\|\widetilde{M}_i^{\tau}\|_2 \leq \|\widetilde{M}_i^{\tau}\|_{1\to 1}
= \max_{S\in \binom{G}{s}} \|\widetilde{M}_i^{\tau}(S, \cdot)\|_1
\leq (\log N)^k \frac{m}{N^{1 - 2/k}};
\eeqn
this bound on the operator norm is what makes the pruned matrices more convenient for us to work with.

We first show that replacing the original matrices by their pruned versions has negligible effect on our bounds.
Indeed, from the definition of~$X_i$ we see that its maximum value is bounded by~$mN$, and so
\begin{align*}
    \big\| M_i^{\tau} - \widetilde{M}_i^{\tau} \big\|_{\infty\to 1}
    &\leq \sum_{S\in \binom{G}{s}} \big\| M_i^{\tau}(S, \cdot) - \widetilde{M}_i^{\tau}(S, \cdot) \big\|_{1} \\
    &\leq 2 \binom{N}{s} \cdot \Exp \big[X_i \,\one\big\{X_i \geq (\log N)^k m N^{-(1-2/k)} \big\}\big] \\
    &\leq 2\binom{N}{s} \cdot mN \Pr \big[X_i \geq (\log N)^k m N^{-(1-2/k)} \big].
\end{align*}
(The multiplication by 2 in the second inequality happens because we must take into account both heavy rows and heavy columns.)
By Lemma~\ref{lem:prune} we conclude that
\beq \label{eq:prune_bound}
\big\| M_i^{\tau} - \widetilde{M}_i^{\tau} \big\|_{\infty\to 1} \leq \frac{2m}{N^3} \binom{N}{s} \quad \text{for all $i\in L$, $\tau\in \bset{R}$.}
\eeq

Next we apply the concentration inequality from Theorem~\ref{thm:Khintchine} to the pruned matrices~$\widetilde{M}_i^{\tau}$;
we obtain
\begin{align*}
    \Exp_{\sigma \in \pmset{L}} \bigg\| \sum_{i\in L} \sigma_i \widetilde{M}_i^{\tau} \bigg\|_{\infty\to 1}
    &\leq \binom{N}{s} \Exp_{\sigma \in \pmset{L}} \bigg\| \sum_{i\in L} \sigma_i \widetilde{M}_i^{\tau} \bigg\|_2 \\
    &\leq 10 \binom{N}{s} \sqrt{\log \binom{N}{s}} \bigg(\sum_{i\in L} \|\widetilde{M}_i^{\tau}\|_2^2 \bigg)^{1/2} \\
    &\leq 10 \binom{N}{s} \sqrt{\log \binom{N}{s}} \bigg(\sum_{i\in L} \|\widetilde{M}_i^{\tau}\|_{1\to 1}^2 \bigg)^{1/2}\\
    &\leq 10 \binom{N}{s} \sqrt{s\log N} \cdot m^{1/2} (\log N)^k \frac{m}{N^{1 - 2/k}}.
\end{align*}
By the triangle inequality and our previous bounds, we conclude that
\begin{align*}
    \Exp_{\sigma \in \pmset{L}} \bigg\| \sum_{i\in L} \sigma_i M_i^{\tau} \bigg\|_{\infty\to 1}
    &\leq \Exp_{\sigma \in \pmset{L}} \bigg\| \sum_{i\in L} \sigma_i \widetilde{M}_i^{\tau} \bigg\|_{\infty\to 1} + \sum_{i\in L} \big\| M_i^{\tau} - \widetilde{M}_i^{\tau} \big\|_{\infty\to 1} \\
    &\leq 10 \binom{N}{s} \sqrt{s\log N} \cdot m^{1/2} (\log N)^k \frac{m}{N^{1 - 2/k}} + \frac{2m^2}{N^3} \binom{N}{s}.
\end{align*}

\subsection{Finishing the proof}

We are now essentially done, and it only remains to combine the upper and lower bounds obtained.
Indeed, combining the last inequality with equation~\eqref{eq:lower_bound} gives
\beqn
\binom{N-4r}{s-2r} m^2 N \ll_{k,\eps} \binom{N}{s} \sqrt{ms\log N} (\log N)^k \frac{m}{N^{1- 2/k}}.
\eeqn
Rearranging and using that $\binom{N}{s}/\binom{N-4r}{s-2r} \ll_k (N/s)^{2r} = N^{2 - 2/k}$, we conclude that
\beqn
m \ll_{k,\eps} s (\log N)^{2k+1} = N^{1-2/k} (\log N)^{2k+1}.
\eeqn
As we started with the assumption~\eqref{eq:Sz_failure}, this shows that
$m_{k-1,\eps}^*(G) \ll_{k,\eps} N^{1-2/k} (\log N)^{2k+1}$
as wished.

\section{Discussion}

Our bounds on~$m_t^*(N)$ are far from the conjectured~$\Theta_t(\log N)$, and we do not believe that they are best possible.
We quickly mention a few avenues that could be explored to obtain better bounds, focusing on the case $m_2^*(N)$ concerning 3-APs for clarity:

\begin{itemize}
    \item A possible source of inefficiency in our arguments is that, after the symmetrization step (Lemma~\ref{lem:symmetrization}), the fact that $D\in G^m$ is a random sequence is not used in any important way.\footnote{It is used in a weak way to obtain the technical conditions~\eqref{eq:distinct} and~\eqref{eq:multiplicity}, but those are mostly technicalities inessential to the main argument.}
    An improvement on our bound for~$m_2^*(N)$ might follow from a possible discrepancy between the worst-case~$D$ considered in the present proof and the average-case setting appearing in the problem.
    
    \item Another possibility is to use a multilinear version of the non-commutative Khintchine inequality to directly bound the final expression in Lemma~\ref{lem:symmetrization} for a fixed sequence~$D$.
    Endow the space of trilinear forms (or tensors) $\R^N\times\R^N\times\R^N\to \R$ with the norm
    \begin{equation*}
        \|T\| = \sup\Big\{|T(x,y,z)|\mid \|x\|_{3},\|y\|_3,\|z\|_3\leq 1\Big\}.
    \end{equation*}
    For trilinear forms $T_1,\dots,T_m$, a bound of the form
    \begin{equation*}
        \Exp_{\sigma\in \pmset{m}}\Big\|\sum_{i=1}^m\sigma_iT_i\Big\| \leq C(N)\Big(\sum_{i=1}^m\|T_i\|^2\Big)^{\frac{1}{2}}
    \end{equation*}
    would imply that $m_2^*(N) \ll C(N)^2$.
    
    The techniques used in the present paper establish the best bounds currently known for \emph{permutation tensors} of the form
    \begin{equation*}
        T(x,y,z) = \sum_{i=1}^Nx_iy_{\pi_1(i)}z_{\pi_2(i)}, 
    \end{equation*}
    where~$\pi_1,\pi_2\in S_N$ are permutations;
    this case is sufficient to deal with the forms~$\Lambda_D$ appearing in our proofs.
    We believe that the bound obtained this way for permutation tensors is not best possible, and sharper bounds for this problem would lead to improvements for~$m_2^*(N)$.
    Note, however, that this avenue by itself does not suffice to obtain a result of the form $m_2^*(N) \leq \poly\log(N)$:
    there is a sequence of permutation tensors (originating from LDC constructions) which imply that $C(N) \geq (\log N)^{\omega(1)}$ is necessary~\cite{BNR:2012, Briet:2016}.

    \item The main technical tool used in the present paper is the non-commutative Khintchine inequality (Theorem~\ref{thm:Khintchine}).
    This inequality is sharp in general, but can be improved when the collection of matrices considered is highly non-commutative;
    see~\cite[Chapter~7]{tropp2015introduction} for a discussion on this point.
    Our matrices are quite close to being commutative, however, and so a possible route for improvement could be to find truly non-commutative matrix embeddings of the multilinear forms $\sum_{x\in G} \prod_{y\in P_{ij}(x)} Z(y)$ appearing in equation~\eqref{eq:large_exp}.
\end{itemize}

\printbibliography

@preamble{{\providecommand{\noopsort}[1]{}}}

@article{tropp2015introduction,
    title={An Introduction to Matrix Concentration Inequalities},
    author={Tropp, Joel A.},
    journal={Foundations and Trends® in Machine Learning},
    volume={8},
    number={1-2},
    pages={1--230},
    year={2015},
    publisher={Now Publishers Inc.},
    doi={10.1561/2200000048},
    url={http://dx.doi.org/10.1561/2200000048}}

@article{BNR:2012,
    AUTHOR = {Bri\"et, Jop and Naor, Assaf and Regev, Oded},
    TITLE = {Locally decodable codes and the failure of cotype for projective tensor products},
    JOURNAL = {Electron. Res. Announc. Math. Sci.},
    FJOURNAL = {Electronic Research Announcements in Mathematical Sciences},
    VOLUME = {19},
    YEAR = {2012},
    PAGES = {120--130},
    ISSN = {1935-9179},
    MRCLASS = {46B07 (46B28)},
    MRNUMBER = {2999057},
    MRREVIEWER = {Daniel\ M.\ Pellegrino},
    DOI = {10.3934/era.2012.19.120},
    URL = {https://doi.org/10.3934/era.2012.19.120}}

@inproceedings{BrietCastroSilva:2023,
    author = {Bri\"{e}t, Jop and Castro-Silva, Davi},
    title = {Raising the roof on the threshold for Szemerédi‘s theorem with random differences},
    year = {2023},
    publisher = {Masaryk University, Brno},
    doi = {10.5817/CZ.MUNI.EUROCOMB23-032},
    series = {EUROCOMB 2023},
    pages = {231-237},
    booktitle = {Proceedings of the 12th European Conference on
Combinatorics, Graph Theory and Applications},
    ISBN = {978-80-280-0344-9},
    EDITOR = {Kr\'{a}l' , Daniel and Ne\v{s}et\v{r}il, Jaroslav}}

@inproceedings{AlrabiahGKM:2022,
    author = {Alrabiah, Omar and Guruswami, Venkatesan and Kothari, Pravesh K. and Manohar, Peter},
    title = {A Near-Cubic Lower Bound for 3-Query Locally Decodable Codes from Semirandom CSP Refutation},
    year = {2023},
    isbn = {9781450399135},
    publisher = {Association for Computing Machinery},
    address = {New York, NY, USA},
    doi = {10.1145/3564246.3585143},
    pages = {1438–1448},
    numpages = {11},
    location = {Orlando, FL, USA},
    series = {STOC 2023}}

@INPROCEEDINGS{BenAroyaRdW:2008,
    author = {Ben-Aroya, Avraham and Regev, Oded and \noopsort{Wolf}{de Wolf}, Ronald},
    booktitle = {2008 IEEE 49th Annual IEEE Symposium on Foundations of Computer Science (FOCS)},
    title = {A Hypercontractive Inequality for Matrix-Valued Functions with Applications to Quantum Computing and LDCs},
    year = {2008},
    volume = {},
    issn = {0272-5428},
    pages = {477-486},
    doi = {10.1109/FOCS.2008.45},
    publisher = {IEEE Computer Society}}

@article{BergelsonL:1996,
    AUTHOR = {Bergelson, V. and Leibman, A.},
    TITLE = {Polynomial extensions of van der {W}aerden's and {S}zemer\'{e}di's theorems},
    JOURNAL = {J. Amer. Math. Soc.},
    FJOURNAL = {Journal of the American Mathematical Society},
    VOLUME = {9},
    YEAR = {1996},
    NUMBER = {3},
    PAGES = {725--753},
    ISSN = {0894-0347},
    MRCLASS = {11B25 (05D10 28D05 54H20)},
    MRNUMBER = {1325795},
    MRREVIEWER = {Pierre Michel},
    DOI = {10.1090/S0894-0347-96-00194-4}}

@article{BollobasT:1987,
    AUTHOR = {Bollob\'{a}s, B. and Thomason, A.},
    TITLE = {Threshold functions},
    JOURNAL = {Combinatorica},
    FJOURNAL = {Combinatorica. An International Journal of the J\'{a}nos Bolyai Mathematical Society},
    VOLUME = {7},
    YEAR = {1987},
    NUMBER = {1},
    PAGES = {35--38},
    ISSN = {0209-9683},
    MRCLASS = {05C80 (04A20 05A05)},
    MRNUMBER = {905149},
    MRREVIEWER = {Zbigniew Palka},
    DOI = {10.1007/BF02579198}}

@article{Bourgain:1987,
    AUTHOR = {Bourgain, J.},
    TITLE = {Ruzsa's problem on sets of recurrence},
    JOURNAL = {Israel J. Math.},
    FJOURNAL = {Israel Journal of Mathematics},
    VOLUME = {59},
    YEAR = {1987},
    NUMBER = {2},
    PAGES = {150--166},
    ISSN = {0021-2172},
    MRCLASS = {11B75 (11B83 11K31)},
    MRNUMBER = {920079},
    MRREVIEWER = {Zun Shan},
    DOI = {10.1007/BF02787258}}

@misc{Briet:2016,
      title={On Embeddings of $\ell_1^k$ from Locally Decodable Codes}, 
      author={Jop Briët},
      year={2016},
      eprint={1611.06385},
      archivePrefix={arXiv},
      primaryClass={cs.CC},
      doi = {doi.org/10.48550/arXiv.1611.06385}}

@article{BrietDG:2019,
    AUTHOR = {Bri\"{e}t, Jop and Dvir, Zeev and Gopi, Sivakanth},
    TITLE = {Outlaw distributions and locally decodable codes},
    JOURNAL = {Theory Comput.},
    FJOURNAL = {Theory of Computing. An Open Access Journal},
    VOLUME = {15},
    YEAR = {2019},
    PAGES = {Paper No. 12, 24},
    MRCLASS = {68Q87 (68R05 68R10 94B05)},
    MRNUMBER = {4028880},
    DOI = {10.4086/toc.2019.v015a012}}

@article{BrietG:2020,
    AUTHOR = {Bri\"{e}t, Jop and Gopi, Sivakanth},
    TITLE = {Gaussian width bounds with applications to arithmetic progressions in random settings},
    JOURNAL = {Int. Math. Res. Not. IMRN},
    FJOURNAL = {International Mathematics Research Notices. IMRN},
    YEAR = {2020},
    NUMBER = {22},
    PAGES = {8673--8696},
    ISSN = {1073-7928},
    MRCLASS = {11B25 (60F10)},
    MRNUMBER = {4216700},
    MRREVIEWER = {Ben Joseph Green},
    DOI = {10.1093/imrn/rny238}}

@misc{Christ:2011,
    doi = {10.48550/ARXIV.1108.5655},
    author = {Christ, Michael},
    title = {On random multilinear operator inequalities},
    publisher = {arXiv},
    year = {2011}}

@article{FrantzikinakisLW:2012,
    AUTHOR = {Frantzikinakis, Nikos and Lesigne, Emmanuel and Wierdl, M\'{a}t\'{e}},
    TITLE = {Random sequences and pointwise convergence of multiple ergodic averages},
    JOURNAL = {Indiana Univ. Math. J.},
    FJOURNAL = {Indiana University Mathematics Journal},
    VOLUME = {61},
    YEAR = {2012},
    NUMBER = {2},
    PAGES = {585--617},
    ISSN = {0022-2518},
    MRCLASS = {37A30 (05D10 28D05)},
    MRNUMBER = {3043589},
    MRREVIEWER = {Lasha Ephremidze},
    DOI = {10.1512/iumj.2012.61.4571}}

@article{FrantzikinakisLW:2016,
    AUTHOR = {Frantzikinakis, Nikos and Lesigne, Emmanuel and Wierdl, M\'{a}t\'{e}},
    TITLE = {Random differences in {S}zemer\'{e}di's theorem and related results},
    JOURNAL = {J. Anal. Math.},
    FJOURNAL = {Journal d'Analyse Math\'{e}matique},
    VOLUME = {130},
    YEAR = {2016},
    PAGES = {91--133},
    ISSN = {0021-7670},
    MRCLASS = {37A45 (05D40 11B30 60F15)},
    MRNUMBER = {3574649},
    MRREVIEWER = {Radhakrishnan Nair},
    DOI = {10.1007/s11854-016-0030-z}}

@article{Furstenberg:1977,
    AUTHOR = {Furstenberg, Harry},
    TITLE = {Ergodic behavior of diagonal measures and a theorem of {S}zemer\'{e}di on arithmetic progressions},
    JOURNAL = {J. Analyse Math.},
    FJOURNAL = {Journal d'Analyse Math\'{e}matique},
    VOLUME = {31},
    YEAR = {1977},
    PAGES = {204--256},
    ISSN = {0021-7670},
    MRCLASS = {10L10 (10K10 28A65)},
    MRNUMBER = {498471},
    MRREVIEWER = {Fran\c{c}ois Aribaud},
    DOI = {10.1007/BF02813304}}

@article{Gowers:2007,
    title={Hypergraph regularity and the multidimensional Szemer{\'e}di theorem},
    author={Gowers, W Timothy},
    journal={Annals of Mathematics},
    pages={897--946},
    volume = {166},
    year={2007},
    publisher={JSTOR},
    doi = {10.4007/annals.2007.166.897 }}

@article{KerenidisW:2004,
    AUTHOR = {Kerenidis, Iordanis and de Wolf, Ronald},
    TITLE = {Exponential lower bound for 2-query locally decodable codes via a quantum argument},
    JOURNAL = {J. Comput. System Sci.},
    FJOURNAL = {Journal of Computer and System Sciences},
    VOLUME = {69},
    YEAR = {2004},
    NUMBER = {3},
    PAGES = {395--420},
    ISSN = {0022-0000},
    MRCLASS = {68P30 (68Q05 81P68 94B65)},
    MRNUMBER = {2087942},
    DOI = {10.1016/j.jcss.2004.04.007},
    note = {Preliminary version in STOC'03}}

@article{KimV:2000,
    AUTHOR = {Kim, Jeong Han and Vu, Van H.},
    TITLE = {Concentration of multivariate polynomials and its applications},
    JOURNAL = {Combinatorica},
    FJOURNAL = {Combinatorica. An International Journal on Combinatorics and the Theory of Computing},
    VOLUME = {20},
    YEAR = {2000},
    NUMBER = {3},
    PAGES = {417--434},
    ISSN = {0209-9683},
    MRCLASS = {05C80 (60C05)},
    MRNUMBER = {1774845},
    MRREVIEWER = {Tomasz J. \L uczak},
    DOI = {10.1007/s004930070014}}

@article {NagleRS:2006,
    AUTHOR = {Nagle, Brendan and R\"{o}dl, Vojt\v{e}ch and Schacht, Mathias},
     TITLE = {The counting lemma for regular {$k$}-uniform hypergraphs},
   JOURNAL = {Random Structures Algorithms},
  FJOURNAL = {Random Structures \& Algorithms},
    VOLUME = {28},
      YEAR = {2006},
    NUMBER = {2},
     PAGES = {113--179},
      ISSN = {1042-9832},
   MRCLASS = {05C35 (05C65)},
  MRNUMBER = {2198495},
MRREVIEWER = {Jozef Skokan},
       DOI = {10.1002/rsa.20117},
       URL = {https://doi.org/10.1002/rsa.20117},
}

@article {RodlS:2004,
    AUTHOR = {R\"{o}dl, Vojt\v{e}ch and Skokan, Jozef},
     TITLE = {Regularity lemma for {$k$}-uniform hypergraphs},
   JOURNAL = {Random Structures Algorithms},
  FJOURNAL = {Random Structures \& Algorithms},
    VOLUME = {25},
      YEAR = {2004},
    NUMBER = {1},
     PAGES = {1--42},
      ISSN = {1042-9832},
   MRCLASS = {05D05 (05C65 05C75 60C05)},
  MRNUMBER = {2069663},
MRREVIEWER = {W. G. Brown},
       DOI = {10.1002/rsa.20017},
       URL = {https://doi.org/10.1002/rsa.20017},
}

@article{Szemeredi:1975,
    AUTHOR = {Szemer\'{e}di, E.},
    TITLE = {On sets of integers containing no {$k$} elements in arithmetic progression},
    JOURNAL = {Acta Arith.},
    FJOURNAL = {Polska Akademia Nauk. Instytut Matematyczny. Acta Arithmetica},
    VOLUME = {27},
    YEAR = {1975},
    PAGES = {199--245},
    ISSN = {0065-1036},
    MRCLASS = {10L10},
    MRNUMBER = {369312},
    MRREVIEWER = {S. L. G. Choi},
    DOI = {10.4064/aa-27-1-199-245}}

@article {TJ:1974,
    AUTHOR = {Tomczak-Jaegermann, Nicole},
     TITLE = {The moduli of smoothness and convexity and the {R}ademacher
              averages of trace classes {$S_{p}(1\leq p<\infty )$}},
   JOURNAL = {Studia Math.},
  FJOURNAL = {Polska Akademia Nauk. Instytut Matematyczny. Studia
              Mathematica},
    VOLUME = {50},
      YEAR = {1974},
     PAGES = {163--182},
      ISSN = {0039-3223},
   MRCLASS = {47B10 (46E30)},
  MRNUMBER = {355667},
MRREVIEWER = {E. Dubinsky},
}

@article{WooleyZ:2012,
    AUTHOR = {Wooley, Trevor D. and Ziegler, Tamar D.},
    TITLE = {Multiple recurrence and convergence along the primes},
    JOURNAL = {Amer. J. Math.},
    FJOURNAL = {American Journal of Mathematics},
    VOLUME = {134},
    YEAR = {2012},
    NUMBER = {6},
    PAGES = {1705--1732},
    ISSN = {0002-9327},
    MRCLASS = {11B30 (11B05)},
    MRNUMBER = {2999293},
    MRREVIEWER = {Martin Klazar},
    DOI = {10.1353/ajm.2012.0048}}

\end{document}